\documentclass{jonasart}

\title{Central products of Cayley-Dickson loops}

\authors{Adam Chapman and Ilan Levin}

\authorinfo[A. Chapman]{Academic College of Tel-Aviv-Yaffo, Israel}{adam1chapman@yahoo.com}

\authorinfo[I. Levin]{Bar-Ilan University, Israel}{ilan7362@gmail.com}

\abstract{This paper studies the triviality of commutators in central products of Cayley-Dickson loops. Two immediate outcomes of this study are (1) the construction of a sequence of non-commutative di-associative loops in which the probability that a random commutator is trivial approaches 1, and (2) an easy proof that if two central products of $n$-fold Cayley-Dickson loops are isomorphic for $n\geqslant 3$, then the loops in the first product are term-wise isomorphic to the loops in the second product.}



\acknowledgments{We thank the anonymous referees whose detailed reports have improved the quality of this paper considerably.}

\VOLUME{1}
\ISSUE{1}
\NUMBER{7}
\DOI{https://doi.org/10.46298/jonas.17802}
\editinfo{March 24, 2026}{June 23, 2026}{Leandro Vendramin}

\begin{document}
	\section{Cayley-Dickson Algebras}

	Given a~field $F$ of $\operatorname{char}(F)\neq 2$ and $\gamma \in F^\times$, a~Cayley-Dickson doubling $B=A\{\gamma\}$ of an algebra $A$ over $F$ with involution $\sigma$ is defined to be $A\oplus A\ell$ where multiplication is given by $(q+r\ell)(s+t\ell)=qs+\gamma\sigma(t)r+(tq+r\sigma(s))\ell$. The involution extends from $A$ to $B$ by $\sigma(q+r\ell)=\sigma(q)-r\ell$. We define the $n$-fold Cayley-Dickson algebra $(\gamma_1,\dots,\gamma_n )_F$ recursively to be $B=A\{\gamma_n\}$ where $A=(\gamma_1,\dots,\gamma_{n-1})_F$, and the initial 0-fold algebra is $F$ with the trivial involution (see~\cite{Schafer:1954}). It is known (see~\cite{Schafer:1954} and~\cite[Section 33.C]{BOI}) that:
	\begin{itemize}
		\item when $n=1$, $(\alpha)_F$ are quadratic \'etale algebras, i.e., either $F \times F$ or quadratic field extensions $K$ of $F$, with the involution being the nontrivial automorphism of the Galois group of $K/F$;
		\item when $n=2$, $(\alpha,\beta)_F$ are quaternion algebras over $F$, with the involution being the symplectic involution;
		\item when $n=3$, $(\alpha,\beta,\gamma)_F$ are octonion algebras;
		\item when $n=4$, $(\alpha,\beta,\gamma,\delta)_F$ are sedenion algebras, etc.
	\end{itemize}
	These algebras, especially for $n\leqslant 3$, are important for various reasons. Tensor products of these algebras, especially for $n=2$ (see~\cite{Merkurjev:1981}), but also for higher $n$ (for instance in~\cite{BlacharRowenVishne:2023}) are also objects of interest in the active research of mathematics. In~\cite{MPP}, isomorphisms between tensor products of octonion algebras were studied. The norm form of $(\gamma_1,\dots,\gamma_n )_F$ given by $x\mapsto x\sigma(x)$ is the $n$-fold Pfister form $\langle \! \langle \gamma_1,\dots,\gamma_n\rangle\!\rangle$ (see~\cite{Musgrave:2015}).

	\section{Cayley-Dickson Loops}

	Recall that a~set $L$ with a~binary operation is called a~\textbf{quasi-group} if the maps $\ell_a : x \mapsto ax$, $r_a : x \mapsto xa$, are bijections, for all $a \in L$. A quasi-group with an identity element $1$ is called a~\textbf{loop}. A loop can be thought of as a~non-associative group. A subloop is a~subset which forms a~loop under the restricted binary operation. In particular, 
    \[
        Z(L)= \left\{z \in L : xz = zx, (xy)z = x(yz), (xz)y = x(zy), (zx)y = z(xy)\right\}
    \]
    is a~subloop. A subloop $N \leqslant L$ is called normal if it is a~kernel of some homomorphism of loops $f: L \to M$. Equivalently, $N$ is normal if and only if $\varphi(N) = N$ for any inner automorphism $\varphi \in \operatorname{Inn}(L)$. Inner automorphisms are the subgroup of $\operatorname{Aut}(L)$ generated by $\operatorname{Inn}(L):= \langle \ell^{-1}_x r_x, r^{-1}_{xy}r_y r_x, \ell^{-1}_{yx}\ell_x\ell_y \, : \, x,y \in L \rangle$. It is easy to see that $\varphi(z)=z$ for any $z \in Z(L)$, hence any subgroup of $Z(L)$ is a~normal subloop of $L$. See~\cite{Bruck:1958} for further background on loops in general.

	For the construction of the loops of concern in this paper, first note that the recursive construction of $(\gamma_1,\dots,\gamma_n )_F$ defines a~set $\ell_1,\dots,\ell_n$ of generators of this algebra over $F$. Fixing a~subgroup $Z$ of $F^\times$ including $\gamma_1,\dots,\gamma_n$, we define the $n$-fold Cayley-Dickson loop $(\gamma_1,\dots,\gamma_n )_Z$ over $Z$ to be the multiplicative subloop of $(\gamma_1,\dots,\gamma_n )_F$ generated by $Z$ and $\ell_1,\dots,\ell_n$ (see~\cite{Kirshtein:2014}). Note that it always contains $-1$. For example, if $Z=\{\pm 1\}$ and $F$ is any field of characteristic $0$, then $(-1,-1)_Z$ is the famous quaternion group $Q_8$, and $(-1,-1,-1)_Z$ is the octonion loop $O_{16}$ (see~\cite{Kirshtein:2012}). We define the central product $L_1 \ast L_2$ of two loops $L_1$ and $L_2$ with center $Z$ in the same manner as in groups, i.e., $L_1 \times L_2$ modulo $\{(x,x^{-1})|x\in Z\}$. The notation is taken from~\cite{Gorenstein:1979}, see also~\cite[Chapter 2, Section 5]{Gorenstein:1980}. As with groups, the center $Z$ can be identified with $\{(z,1): z\in Z\}$ in $L_1 \ast L_2$, and this is then the center of $L_1 \ast L_2$. As with groups (see for example in~\cite{Gorenstein:1979}, this notion extends without any problem to central products of any finite number of loops sharing the same center $Z$. Note that the central product $L_1 \ast \dots \ast L_m$ of Cayley-Dickson loops $L_i =(\gamma_{i,1},\dots,\gamma_{i,n})_Z$, $i \in \{1,\dots,m\}$, is also the multiplicative subloop of the corresponding tensor product $A_1\otimes_F \dots \otimes_F A_m$ over $F$ of the Cayley-Dickson algebras $A_i =(\gamma_{i,1},\dots,\gamma_{i,n})_F$ generated by the generators of these algebras and $Z$. All the loops mentioned above are di-associative (that is, every two elements generate an associative subgroup, see~\cite[Theorem 5.4]{Culbert:2007}), so they are ``close'' to being groups in some sense.

	The goal of this paper is to examine how differently central products of loops behave from groups. We focus on Cayley-Dickson loops, but this calls for a~deeper study of larger families of such loops. For the group of quaternions $Q_8 =(-1,-1)_Z$ with $Z=\{\pm 1\}$, we have $Q_8 \ast Q_8 =(1,-1)_Z \ast (1,-1)_Z$, even though $Q_8 \not \cong (1,-1)_{Z}$. We show that this does not extend to higher-fold Cayley-Dickson loops. Another important difference is the commutativity degree: For non-abelian groups, it cannot exceed $\frac{5}{8}$. We show here that for central products of Cayley-Dickson loops, the commutativity degree can approach 1. More precisely, we prove the following statements:
	\begin{enumerate}
		\item The probability that a~random commutator in $L_1 \ast L_2$ is trivial approaches 1 as $n$ approaches $\infty$ when $L_1$ and $L_2$ are $n$-fold Cayley-Dickson loops over $Z$ of $|Z|<\infty$.
		\item If $D_1,\dots,D_m,E_1,\dots,E_m$ are $n$-fold Cayley-Dickson loops over $Z$ such that $n \geqslant 3$ and $D_1\ast \dots \ast D_m \cong E_1 \ast \dots \ast E_m$, then there exists $\sigma \in S_m$ such that $D_k \cong E_{\sigma(k)}$ for all $k \in \{1,\dots,m\}$.
	\end{enumerate}
	See~\cite{Levin2025AssociativeMoufang} for a~similar discussion on the associativity degree of Moufang loops. Item (2) is related to questions about automorphisms of Cayley-Dickson loops, as were studied in~\cite{Kirshtein:2012} and~\cite{Chapman2024LoopsInvolutionCayleyDickson}, and serves as a~loop analogue (at least for $n = 3$) of~\cite{MPP}. We also study the associativity degree of finite $n$-fold Cayley-Dickson Loops $(-1,\dots,-1)_Z$.

	\section{Commutants and Commutators}

	Given a~loop $L$ and $x \in L$, define the commutant of $x$ as $C_L (x):= \{y\in L|xy=yx\}$. The commutator $[x,y]$ of $x,y \in L$ is the unique element such that $(xy) = [x,y](yx)$, whose existence and uniqueness follow from the definition of a~loop. Note that in our case, where all the loops are central products of Cayley-Dickson loops, all the commutators are in $\{\pm 1\}$.

	Let $D_1,\dots,D_m$ be $n$-fold Cayley-Dickson loops over $Z$. Note that each $D_i$ possesses a~surjective homomorphism from $D_i$ to $(\mathbb{Z}/2\mathbb{Z})^n$ sending the $k$-th generator of $D_i$ to $e_k$ for all $k = 1, \dots, n$, and the elements of $Z$ to $\vec{0}$. The kernel is thus $Z$. Hence, $D_i /Z\cong (\mathbb{Z}/2\mathbb{Z})^n$. For $A=D_1 \ast \dots \ast D_m$, we obtain in a~similar way an isomorphism $A/Z \cong (\mathbb{Z}/2\mathbb{Z})^{mn}$. The finiteness of $A/Z$ as a~set enables us to carry out the following calculations.

	Note that for any positive integer $d$, $G=(\mathbb{Z}/2\mathbb{Z})^d$ is both an additive group and an $\mathbb{F}_2$-vector space, and every subgroup of $G$ is thus an $\mathbb{F}_2$-vector space as well which has a~well-defined dimension. We apply this notion in the proofs below.

	\begin{proposition}\label{Cardinality}
		Let $D_1,\dots,D_m$ be $n$-fold Cayley-Dickson loops over $Z$ and let $x \in D_k \setminus Z$ for some $k \in \{1,\dots,m\}$. Set $A=D_1 \ast \dots \ast D_m$. Then, $C_A (x)/Z$ is of cardinality $2^{(m-1)n+1}=\frac{1}{2^{n-1}}2^{mn}=\frac{1}{2^{n-1}}\cdot|A/Z|$.
	\end{proposition}

	\begin{proof}
		This readily follows from the fact that $C_{D_k}(x)/Z$ is of dimension $1$, $D_k /Z$ is of dimension $n$ and the definition of the central product.
	\end{proof}

	We say that an element $y \in A$ is of rank $k$ if:
	\begin{enumerate}
		\item $y = x_{i_1}\cdots x_{i_k}$;
		\item $x_{i_j} \in D_{i_j}\setminus Z$ for all $j = 1,\dots,k$ and;
		\item $i_1 < i_2 < \dots < i_k$.
	\end{enumerate}
	We denote it by $\operatorname{rank}_A (y) = k$. If $y \in Z$, then $\operatorname{rank}_A (y):= 0$. Note that this is well-defined, as we can omit brackets and change the order of the $x_i$'s as we wish, because all their associators and commutators vanish.

	\begin{theorem}\label{recurrence}
		Let $D_1,\dots,D_m$ be $n$-fold Cayley-Dickson loops over $Z$, let $y \in A$, and let $\operatorname{rank}_A (y) = k$. Then $\frac{|C_A (y)/Z|}{2^{mn}}=\frac{1}{2}+\frac{1}{2}(2p-1)^k$, where $p=\frac{1}{2^{n-1}}$.
	\end{theorem}

	\begin{proof}
		Set $b_k =\frac{1}{2^{mn}}|C_A (y)/Z|$ for any $y \in A$ of $\operatorname{rank}_A (y) = k$

		Let us prove by induction on $k$ that $b_k =\frac{1}{2}+\frac{1}{2}(2p-1)^k$. The base case is covered by Proposition~\ref{Cardinality}. Assume it is true for $k-1$. Write $y = x_{i_1}\cdots x_{i_k}$. By the definition of the central product, for any $\tilde{a} \in A$, the commutator $[y,\tilde{a}]$ is the product of the different commutators $[x_{i_1},\tilde{a}],\dots,[x_{i_{k-1}},\tilde{a}],[x_{i_k},\tilde{a}]$. As a~result, we obtain the identity
		\[
            [y,\tilde{a}] 
            = [x_{i_1}\cdots x_{i_{k-1}}, \tilde{a}] \cdot [x_{i_k},\tilde{a}].
        \]
        Therefore, $\tilde{a} \in C_A (y)$ if and only if 
        \[
            [x_{i_1}\cdots x_{i_{k-1}},\tilde{a}] 
            = [x_{i_k},\tilde{a}] = 1
            \quad \text{or} \quad
            [x_{i_1}\cdots x_{i_{k-1}},\tilde{a}] 
            = [x_{i_k},\tilde{a}] = -1.
        \]
        The probability of the former is $pb_{k-1}$ and the probability of the latter is $(1-p)(1-b_{k-1})$. So $b_k = pb_{k-1}+(1-p)(1-b_{k-1}) = (1-p)+(2p-1)b_{k-1}$, and by plugging the induction assumption $b_{k-1} = \frac{1}{2}+\frac{1}{2}(2p-1)^{k-1}$ we are done.
	\end{proof}

	When $|Z|<\infty$, one can see that
	\begin{equation}\label{rank}
		\tag{$\clubsuit$} |\left\{x \in A \, : \, \operatorname{rank}_A (x)=k\right\}| = |Z|\cdot \binom{m}{k}(2^{n}-1)^k.
	\end{equation}

	\begin{corollary}
		Let $(x,y)$ be an element chosen from $L\times L$ with uniform distribution, where $L=D_1 \ast D_2$ is the central product of two $n$-fold Cayley-Dickson loops over $Z$ with $|Z| < \infty$. Then $\mathbb{P}_c (L) = \mathbb{P}(xy=yx)=1-6\cdot\frac{1}{2^n}+22\cdot\frac{1}{4^n}-24\cdot\frac{1}{8^n}+8\cdot \frac{1}{16^n}$.
	\end{corollary}

	\begin{proof}
		Note that $|L| = |Z|\cdot2^{2n}$. By the law of total probability with respect to $\operatorname{rank}_A (x)$, we have
		\[
			\mathbb{P}(xy=yx) = \sum_{i=0}^{2} \mathbb{P}(xy=yx | \operatorname{rank}_A (x)=i) \cdot \mathbb{P}(\operatorname{rank}_A (x)=i).
		\]
		Moreover, using~\eqref{rank} for $m=2$, the distribution of $\operatorname{rank}_A (x)$ is given by
		\[
			\mathbb{P}(\operatorname{rank}_A (x)=i)=\frac{|Z|}{|A|}\binom{2}{i}(2^n -1)^i.
		\]
		Substituting these values yields
		\begin{gather*}
			\mathbb{P}(xy=yx) = \frac{|Z|}{|A|}(1+p \cdot2\cdot (2^n -1)+(\frac{1}{2}+\frac{1}{2}(2p-1)^2 )\cdot (2^n -1)^2 )= \\
			= 1-6\cdot\frac{1}{2^n}+22\cdot\frac{1}{4^n}-24\cdot\frac{1}{8^n}+8\cdot \frac{1}{16^n}
		\qedhere
		\end{gather*}
	\end{proof}

	\begin{corollary}
		There exists a~sequence of di-associative loops $\left\{\mathcal{K}_n\right\}_{n=1}^{\infty}$ such that the commutativity degree approaches 1, i.e., $\lim_{n\rightarrow \infty}\mathbb{P}_c (\mathcal{K}_n )=~1$.
	\end{corollary}

	\begin{proof}
		Take $\mathcal{K}_n$ to be as in the previous corollary and $n \rightarrow \infty$.
	\end{proof}

	This demonstrates how versatile di-associative loops are, even rather natural ones, compared to groups in terms of their commutativity degree.

	\begin{remark}
		If we consider $D_1,\dots,D_m$ to be $n$-fold Cayley-Dickson loops over $Z$ and $A_m = D_1 \ast \cdots \ast D_m$, fix $n \geqslant 2$, and then we get $\lim_{m\rightarrow \infty}\mathbb{P}_c (A_m ) = \frac{1}{2}$. This was known for $n=2$, in which case these loops are quaternionic groups.
	\end{remark}

	We now turn our concern to understanding automorphisms of these objects, when $n \geqslant 3$.

	\begin{theorem}
		If $D_1,\dots,D_m,E_1,\dots,E_m$ are $n$-fold Cayley-Dickson loops over $Z$ with $n \geqslant 3$ and $D_1 \ast \dots \ast D_m \cong E_1 \ast \dots \ast E_m$, then there exists $\sigma \in S_m$ for which $D_j \cong E_{\sigma(j)}$.
	\end{theorem}

	\begin{proof}
		Let $\varphi: D_1 \ast \dots \ast D_m \rightarrow E_1 \ast \dots \ast E_m$ be an isomorphism. Given $b_k$ as in Theorem~\ref{recurrence}, $|b_k -\frac{1}{2}| = \frac{1}{2}|2p-1|^{k}$ is strictly decreasing, since $p = \frac{1}{2^{n-1}}$ and $n \geqslant 3$. Hence, $\left\{b_k\right\}_{k=0}^{\infty}$ has no repetitions. Thus, $\varphi$ preserves the rank. Let $x \in D_j$, $\operatorname{rank}_D (x)=1$, for some $j \in \{1,\dots, m\}$. Then $\varphi(x) \in E_\ell$, for some $\ell$, because its rank is $1$. The only elements anti-commuting with $x$ are $D_j\setminus Z\langle x \rangle$, and the only elements anti-commuting with $\varphi(x)$ are $E_\ell \setminus Z\langle \varphi(x) \rangle$, so $\varphi(D_j\setminus Z\langle x \rangle) = E_\ell \setminus Z\langle \varphi(x) \rangle$, and altogether we get that $\varphi(D_j )=E_\ell$. So there exists $\sigma \in S_m$ for which $D_j \cong E_{\sigma(j)}$.
	\end{proof}

	\section{Associativity Degree}

	In~\cite{Levin2025AssociativeMoufang} it was shown that the associativity degree of a~finite non-associative Moufang loop is bounded from above by $\frac{43}{64}$. Here we study the associativity degree of the $n$-fold Cayley Dickson loops $L_n =(-1,\dots,-1)_Z$, for $|Z| < \infty$. It is important to note that the order of an $n$-fold Cayley-Dickson loop is $2^n |Z|$. In this context, we define the associativity degree of $L_n$, denoted by $\mathbb{P}_{\mathcal{A}}(L_n )$, to be the probability that for a~random triple of elements $(x,y,z)$ chosen from $L_n \times L_n \times L_n$ with uniform distribution, $\langle x,y,z \rangle $ is a~group. This coincides with the definition in~\cite{Levin2025AssociativeMoufang} for Moufang loops, where this is equivalent to saying that $(xy)z=x(yz)$ because of Moufang's theorem.

	\begin{theorem}
		The associativity degree of $L_n$ is $\mathbb{P}_\mathcal{A}(L_n ) = \frac{7\cdot 4^n -14 \cdot 2^n +8}{8^n}$.
	\end{theorem}

	\begin{proof}
		Let $(a,b,c)$ be an arbitrary element in $L_n \times L_n \times L_n$. It amounts to counting the number of triples generating a~subgroup. If $a \in Z$, then $b$ and $c$ can be any elements. We have $2^{2n}\cdot |Z|^3$ such triples. If $a \not \in Z$ (and there are $2^n\cdot |Z|-|Z|$ such elements), we consider two cases: $b \in \langle a, Z \rangle$ and $b \not \in \langle a, Z \rangle$. Note that $|\langle a, Z \rangle| = 2|Z|$. If $b\in \langle a, Z \rangle$, then $c$ can be any element, and there are $|Z|(2^n -1)\cdot 2|Z| \cdot 2^n |Z|$ such triples. If $b \not \in \langle a, Z \rangle$, then $c$ has to belong to $\langle a,b, Z \rangle$ by~\cite[Lemma 7]{Kirshtein:2014}, and there are $|Z|(2^n -1)\cdot |Z|(2^n -2) \cdot 4|Z|$ such triples. After summing everything and dividing by $|L_n |^3$, we finally get that 
        \[
            \mathbb{P}_\mathcal{A}(L_n)
            = \frac{7\cdot 4^n -14 \cdot 2^n +8}{8^n}.
        \qedhere
        \]
	\end{proof}

	For example, when $n=2$, $\frac{7\cdot 4^2 -14 \cdot 2^2 +8}{8^2}=\frac{64}{64}=1$, as expected, and when $n=3$, $\frac{7\cdot 4^3 -14 \cdot 2^3 +8}{8^3}=\frac{344}{512}=\frac{43}{64}$, as was already found in~\cite{Levin2025AssociativeMoufang}.

	%%% REFERENCES %%%
	
\end{document}